\documentclass{base_class}

\title[]{On mixed Hodge-Riemann relations for translation-invariant valuations and   Aleksandrov-Fenchel inequalities} 
\author{Jan Kotrbat\'y}
\author{Thomas Wannerer}
\email[Corresponding author]{kotrbaty@math.uni-frankfurt.de}
\email{thomas.wannerer@uni-jena.de}
\address{Institut für Mathematik, Goethe-Universit\"at Frankfurt, Robert-Mayer-Str. 10, 60629 Frankfurt, Germany}
\address{Fakult\"at f\"ur Mathematik und Informatik, Friedrich-Schiller-Universit\"at Jena, 07743 Jena, Germany}

\thanks{JK was supported by FWF Grant P31448-N35 and by DFG grant BE 2484/5-2. TW was supported by DFG grant WA 3510/3-1.}
\date{\today}

\begin{document}

\begin{abstract}

A version of the Hodge-Riemann relations for valuations was recently conjectured and proved in several special cases by the first-named author \cite{Kotrbaty:HR}. 
The Lefschetz operator considered there arises as either the product or the convolution with the mixed volume of several Euclidean balls. Here we prove that in (co-)degree one the Hodge-Riemann relations persist if the balls are replaced by several different (centrally symmetric) convex bodies with smooth boundary with positive Gauss curvature.
While these mixed Hodge-Riemann relations for the convolution directly imply the Aleksandrov-Fenchel inequality, they yield for the dual operation of the product  a new inequality. This new inequality
strengthens classical consequences of  the Aleksandrov-Fenchel inequality for lower dimensional convex bodies and generalizes some of the geometric inequalities recently discovered by S.~Alesker \cite{Alesker:Kotrbaty}.

\end{abstract}

\maketitle

\section{Introduction}

\subsection{General background}

In convex geometry,  a valuation is a function on the space of convex bodies that is finitely additive.
 The volume  and, more generally, the  mixed volume  of convex bodies are the prime examples of valuations. In fact,  according to a conjecture of P. McMullen, the mixed volumes span a dense subspace of the space of continuous and translation-invariant valuations.

Systematically investigated already by W.~Blaschke and H.~Hadwiger,  valuations have  attracted interest from various branches of convexity ever since.   
A major rupture in this development were the  proof  of McMullen's conjecture 
and  the many new ideas introduced  by S.~Alesker \cite{Alesker:Irreducibility} in 2001. The subsequent discovery, also pioneered by Alesker, of various algebraic structures on valuations, led to striking developments in integral geometry.  The key insight due to J.H.G.~Fu \cite{Fu:Unitary} is that  the kinematic formulas in an isotropic space $(M,G)$ are, in a precise sense, dual to the Alesker product of $G$-invariant valuations on $M$. A landmark result obtained by harnessing the power of the algebraic structures on  valuations is the determination of the kinematic formulas in complex spaces forms by A.~Bernig,  J.H.G.~Fu, and G.~Solanes \cite{BernigFu:Hig,BFS}. For  more results and intriguing phenomena arising from this new algebraic perspective, we refer the reader to \cite{AbardiaBernig:AdditiveFlag, Bernig:Tutte,BernigHug:Tensor,FuW:Riemannian,Solanes:Contact,SolanesW:Spheres}, the lecture notes \cite{AleskerFu:Barcelona},  and the numerous references therein.
 Concerning other, quite different perspectives on valuations we mention the recent articles \cite{JochemkoSanyal:Combinatorial, CLM:Hadwiger,CLM:Homogeneous,BFS:PseudoRiemannian,Faifman:Contact}.

 Inasmuch as McMullen's conjecture is true, it seems reasonable to expect that the new algebraic structures could shed a new light even on the classical notion of mixed volume. While indeed this apparatus turned out to be remarkably useful to prove identities for valuations arising for example in integral geometry,  it seemed to have  nothing to offer concerning inequalities.
Very recently, however, this story has  taken an abrupt turn; namely, the first-named author \cite{Kotrbaty:HR} conjectured   and proved in several special cases an analogue for valuations of the Hodge-Riemann  relations from K\"ahler geometry.  Remarkably, as first observed by S.~Alesker \cite{Alesker:Kotrbaty},   these  results directly imply both classical and  new geometric inequalities between the mixed volumes of convex bodies.

\subsection{Main results}
In K\"ahler geometry, the Hodge-Riemann relations hold not only for a power of one single Lefschetz operator, but there is also a mixed version where the Lefschetz map may be composed of several operators chosen from a certain cone. First stated and proved in special cases by M.~Gromov \cite{Gromov:ConvexSets} and by V.A.~Timorin \cite{Timorin:Mixed}, these mixed Hodge-Riemann relations have been proved in full generality by T.-C.~Dinh and V.-A.~Nguy\^en  \cite{DinhNguyen:Mixed}  in  2006. Concerning other incarnations of the Hodge-Riemann relations we mention the breakthrough results \cite{McMullen:SimplePolytopes,AHK:Matroids} and refer the reader to the survey by J.~Huh \cite{Huh:HR}.

 There are two natural multiplicative structures on valuations: the Alesker product and the Bernig-Fu convolution.
The first-named author  formulated in  \cite[Conjecture E]{Kotrbaty:HR} a  mixed version of the  Hodge-Riemannian  relations for the convolution and  showed that if true it would imply the  Aleksandrov-Fenchel inequality. In this note we prove this conjecture in co-degree $1$ (Theorem~\ref{thm:MixedConvolution}) and deduce from it a dual version in degree $1$ for the Alesker product (Theorem~\ref{thm:MixedProduct}). From the latter result we further deduce an apparently new geometric inequality for mixed volumes that  is formally analogous to the Aleksandrov-Fenchel inequality (Theorem \ref{thm:DAF}). Corollary 1.5 and, more generally, Theorem~\ref{thm:AFfunctionals} show that this inequality improves  classical consequences of the Aleksandrov-Fenchel inequality for lower dimensional convex bodies.

 Let $V(A_1,\ldots, A_n)$  denote the mixed volume of the convex bodies $A_1,\ldots, A_n\subset\RR^n$. We denote  by $\Val^\infty_i$ the subspace of $i$-homogeneous smooth translation-invariant valuations and  by $*$ the convolution of valuations. For precise definitions and further background on valuations, see Section~\ref{sec:prelim}.
We call a valuation $\phi$ non-positive and  write $\phi\leq 0$ if   $\phi(A)\leq 0$ for all convex bodies $A$. Recall also that the complex conjugate of a valuation is by definition $ \overline \phi(A)= \overline{\phi(A)}$.

\begin{theorem}\label{thm:MixedConvolution} Let $C_1,\ldots, C_{n-1} \subset \RR^n$ be convex bodies with smooth boundary with positive Gauss curvature and consider the valuations
$$\psi_i(A)= V(A[n-1], C_i),\quad i=1,\dots,n-1.$$
Then the following properties hold:
 \begin{enuma}
  \item Hard Lefschetz theorem: The map $\Val_{n-1}^\infty \to \Val^\infty_{1}$ defined by 
  \begin{equation}\label{eq:LConvolution}\phi \mapsto  \phi * \psi_1*\cdots *\psi_{n-2}\end{equation}
  is an isomorphism of topological vector spaces.
  \item Hodge-Riemann relations: If $\phi\in\Val_{n-1}^\infty $ is co-primitive in the sense that 
  $$ \phi * \psi_1*\cdots *\psi_{n-1} =0,$$
  then 
  $$  \overline{ \phi} * \phi  * \psi_1*\cdots * \psi_{n-2}\leq 0$$
  and equality holds if and only if $\phi=0$. 
 \end{enuma}

\end{theorem}

 It was shown in \cite[Corollary 8.1]{Kotrbaty:HR} that item (b) of Theorem~\ref{thm:MixedConvolution} directly implies the Aleksandrov-Fenchel inequality. Here we show that conversely Aleksandrov's proof of the Aleksandrov-Fenchel inequality via elliptic-operator theory \cite{Aleksandrov:Theorie4} combined with some rather standard arguments from valuation theory already implies  Theorem~\ref{thm:MixedConvolution}.   This adds to the renewed interest Aleksandrov's proof has attracted in recent years, see, e.g., the works by J.~Abardia and the second-named author \cite{AW:Inequalities}, A.V.~Kolesnikov and E.~Milman \cite{KolesnikovMilman:Lp}, and Y.~Shenfeld and R.~van Handel  \cite{ShenfeldHandel:Bochner}.
That Aleksandrov's proof implies   Theorem~\ref{thm:MixedConvolution} has also been observed independently by S.~Alesker  (private communication).

The Alesker-Fourier transform  is a natural isomorphism of the space of smooth and translation-invariant valuations that intertwines the product and convolution, see Section~\ref{sec:prelim} below.  In this precise sense the two multiplicative structures on valuations are  thus dual to each other.  Using the Alesker-Fourier transform, we deduce from Theorem~\ref{thm:MixedConvolution} the corresponding statements for the product.
The Euler-Verdier involution $\sigma$ appearing below is defined by 
$$ (\sigma \phi)(A) = (-1)^i \phi(-A)$$
if $\phi\in\Val_i$.
We call a valuation $\phi$ non-negative and  write $\phi\geq 0$ if   $\phi(A)\geq 0$ for all convex bodies  $A$.

\begin{theorem}\label{thm:MixedProduct}
Let $C_{i,j}\subset \RR^n$, $i,j=1,\dots, n-1$, be centrally symmetric  convex bodies with smooth boundary with positive Gauss curvature and consider the valuations $$\psi_i(A)= V(A, C_{i,1},\ldots, C_{i,n-1}),\quad i=1,\dots,n-1.$$
Then the following properties hold:
 \begin{enuma}
  \item Hard Lefschetz theorem: The map $\Val_1^\infty \to \Val^\infty_{n-1}$ defined by 
  $$\phi \mapsto \phi \cdot \psi_1\cdots \psi_{n-2}$$
  is an isomorphism of topological vector spaces.
  \item Hodge-Riemann relations: If $\phi\in\Val_1^\infty $ is primitive in the sense that 
  $$ \phi \cdot \psi_1\cdots \psi_{n-1} =0,$$
  then 
  $$\overline{\sigma \phi}  \cdot  \phi  \cdot \psi_1\cdots \psi_{n-2}\geq 0$$
  and equality holds if and only if $\phi=0$. 
 \end{enuma}

\end{theorem}

Just like Theorem~\ref{thm:MixedConvolution} implies the Aleksandrov-Fenchel inequality, it is natural to ask whether its dual version Theorem~\ref{thm:MixedProduct} can as well be given a geometric meaning.  To this end, inspired by the geometric inequalities recently discovered by S.~Alesker \cite{Alesker:Kotrbaty} we introduce in Definition~\ref{def:higherMixed} below  an array of new numerical invariants for tuples of convex bodies. Their elementary properties are analogous to those of the usual mixed volume (see Proposition~\ref{prop:higherMixed}) which is, in fact, included as a special case.  From this point of view it seems justified to regard these quantities  as \emph{higher rank mixed volumes}.

In order to emphasize that our definition does not depend on  the choice of inner product, we work in an abstract vector space.  In this connection, observe that the Euclidean structure is  also completely irrelevant for Theorems~\ref{thm:MixedConvolution} and \ref{thm:MixedProduct}; for a metric-free  formulation of these results see Remarks~\ref{rem:Metricfree1} and \ref{rem:Metricfree2} below.

\begin{definition}
\label{def:higherMixed}
 Let $W$ be an $n$-dimensional real vector space with a fixed positive Lebesgue measure $\vol_W$  and let $l\in\{2,\dots,n\}$. Let $\Delta_l\colon W\to W^l$ be the diagonal embedding and let $\vol_{\coker \Delta_l}$ be the Lebesgue measure on $\coker \Delta_l$ induced by the exact sequence 
 $$ 0 \longrightarrow 
 \im \Delta_l\longrightarrow W^l \longrightarrow \coker \Delta_l \longrightarrow 0,$$
(see Section~\ref{sec:higherMixed} for details).
 For $i=1,\dots,l$, consider the map $f_i\colon W\to \coker \Delta_l$ given by the composition of the inclusion 
 $W\hookrightarrow W^l$  into the $i$-th  summand with the canonical map $W^l\to \coker \Delta_l$. 
 
Let $n=k_1+\cdots +k_l$ be a partition into positive integers. Given $(n-k_i)$-tuples $\calA_i=(A_{i,1},\ldots,A_{i,n-k_i})$, $i=1,\dots,l$, of convex bodies, we define their \emph{mixed volume of rank $l-1$ (corresponding to the partition $n=k_1+\cdots +k_l$)} by 
 \begin{align}
 \label{eq:higherMixed}
 \begin{split}
  &\widetilde V(\calA_1,\ldots,\calA_l) \\
&\qquad= V_{\coker \Delta_l}\big(f_1(A_{1,1}), \ldots , f_1(A_{1,n-k_1}),  \ldots, f_l(A_{l,1}), \ldots,  f_l(A_{l,n-k_l})\big).
  \end{split}
 \end{align}
\end{definition}

We will see in Proposition~\ref{prop:higherMixedUsual} that the usual mixed volume is  essentially just the mixed volume of rank    $1$; in fact, it is recovered for any next-to-minimal length partition  $n=k+(n-k)$. In the present article, however, our focus lies on the opposite margin of Definition \ref{def:higherMixed}, namely, on  rank $n-1$ corresponding to a single partition $n=1+\cdots+1$. In certain special cases, the corresponding higher rank mixed volume admits a simple geometric description in terms of the  usual mixed volume of convex bodies in the dual space $W^*$ (see Proposition~\ref{prop:higherMixedSpecial}), but in general we are not aware of such a reduction.

Our third main result is that  the mixed volume of rank  $n-1$ satisfies an inequality that is formally analogous to the Aleksandrov-Fenchel inequality. In fact, even an analogue of the determinantal improvement of the Aleksandrov-Fenchel inequality which is due to G.C.~Shephard \cite{Shephard:MixedVolumes} holds as follows.

\begin{theorem}\label{thm:DAF}
  Let $W$ be an $n$-dimensional real vector space with a fixed positive  Lebesgue measure. Let $\calA_i$, $i=1,\ldots,m$, and $\calC_i$, $i=1,\ldots,n-2$, be $(n-1)$-tuples of convex bodies in  $W$. Assume further that the bodies in $\calA_m$ and $\calC_i$, $i=1,\dots,n-2$, are centrally symmetric. Then the $m\times m$ matrix $(v_{ij})$ given by
 $$ v_{ij}= \widetilde V( \calA_i, - \calA_j,  \calC_{1},\ldots, \calC_{n-2}),$$
 where $ - \calA_j= (-A_{j,1},\ldots,-A_{j,n-1})$, satisfies
 $$ (-1)^m \det(v_{ij}) \leq 0.$$
\end{theorem}

 Observe that for $m=2$ Theorem~\ref{thm:DAF} yields the inequality
\begin{align}\label{eq:daf2}\begin{split}  &\widetilde V( \calA_1, - \calA_2,  \calC_{1},\ldots, \calC_{n-2})^2 \\
 &\qquad \geq \widetilde V( \calA_1, - \calA_1,  \calC_{1},\ldots, \calC_{n-2}) \widetilde V( \calA_2, - \calA_2,  \calC_{1},\ldots, \calC_{n-2}).\end{split}
\end{align}
 It  may be  illuminating to express special cases of  \eqref{eq:daf2} in more familiar terms. 
If $\dim W=2$, then \eqref{eq:daf2} is equivalent to the Aleksandrov-Fenchel inequality in the plane where  one of the two bodies is centrally symmetric (see Proposition~\ref{prop:higherMixedUsual}).
 If $\dim W=3$, however,  choosing $C_{1,1}=C_{1,2}$ to be $2$-dimensional, we obtain  the following, apparently new inequality between mixed volumes.

\begin{corollary}\label{cor:R3} Let $W$ be a $3$-dimensional real vector space with a fixed positive Lebesgue measure. For a linear functional  $\xi\in W^*$, consider the graphing map $\b\xi\maps{ W}{W\times \RR}$, $\b\xi(w)=\big(w,\xi(w)\big)$.
Then for all convex bodies $A_1,A_2, B_1,B_2\subset W$ with $B_1,B_2$ centrally symmetric  we have
\begin{align*}  V \big(A_1,A_2,  \b\xi (B_1), &\b\xi(B_2)\big)^2  \geq  V \big(A_1,A_2, \b\xi (A_1),\b\xi(A_2)\big) V \big(B_1,B_2, \b\xi (B_1),\b\xi(B_2)\big).
 \end{align*}
\end{corollary}

It is a consequence of the Aleksandrov-Fenchel inequality that
\begin{equation}\label{eq:afR4} V(A,A,B,B)^2\geq   V(A,A,A,A) V(B,B,B,B)\end{equation}
 holds for all convex bodies $A,B$   in a $4$-dimensional space. 
If $A$ and $B$ lie in hyperplanes, then   the right-hand side of \eqref{eq:afR4} is zero and the inequality becomes trivial. From this perspective, Corollary~\ref{cor:R3} may be viewed as a  strengthening of  a classical consequence of the Aleksandrov-Fenchel inequality for lower dimensional convex bodies.   For a generalization of  Corollary~\ref{cor:R3} to an arbitrary dimension see Theorem~\ref{thm:AFfunctionals} below.  We do not aim at developing this any further here, but let us mention that mixed volumes of lower-dimensional convex bodies  arise in interesting combinatorial applications of the Aleksandrov-Fenchel inequality \cite{ShenfeldHandel:Extremals,Stanley:TwoApplications}.

Let us  briefly comment on the relationship between the geometric inequalities above and the ones recently discovered by S.~Alesker.  First, inequality  \eqref{eq:daf2} directly  implies  \cite[Corollary~3.6]{Alesker:Kotrbaty}.  Second, \cite[Theorem 1.1, item (2)]{Alesker:Kotrbaty} easily follows by applying \eqref{eq:daf2} to the Blaschke sum $A\#(-A)$. 
Finally, \cite[Theorem 1.1, item (1)]{Alesker:Kotrbaty} is a special case of Theorem~\ref{thm:DMVminus}.

\subsection*{Acknowledgements}  Thanks are due to S.~Alesker, A.~Bernig, E.~Lutwak, and V.~Milman for their helpful comments.

\section{Preliminaries} \label{sec:prelim}

For the convenience of the reader we collect here several facts from valuation theory and convex geometry that will be needed later.  

\subsection{Valuation theory} The standard references for this material are Chapter 6 of  Schneider's book \cite{Schneider:BM}, the book \cite{Alesker:Kent} and the  lecture notes \cite{AleskerFu:Barcelona} by Alesker, and the original papers on the product, convolution, and Fourier transform of valuations \cite{Alesker:Product,BernigFu:Convolution,Alesker:Fourier}. 

\subsubsection{Continuous and smooth valuations}

 Let $W$ be  an $n$-dimensional real vector space and let $\calK(W)$ denote the space of convex bodies, i.e., non-empty compact convex subsets, in $W$. 
A valuation is  a function $\phi\colon \calK(W)\to \CC$ such that
$$ \phi(A\cup B) = \phi(A)+ \phi(B)-\phi(A\cap B)
$$
 holds for any $A,B\in\calK(W)$ whenever $A\cup B$ is convex. We denote by $\Val(W)$ the space of translation-invariant continuous valuations, equipped with the topology of uniform convergence on compact subsets. If no confusion can  arise, we  abbreviate $\Val=\Val(W)$.  According to McMullen's decomposition theorem,
$$\Val(W)= \bigoplus_{i=0}^n  \Val_i(W),$$
where $\Val_i(W)$ is the subspace of $i$-homogeneous valuations. This grading may be further refined by considering  even and odd valuations with respect to the reflection in the origin, i.e., satisfying
$\phi(-A)= \phi(A)$ and $\phi(-A)=-\phi(A)$, respectively.

The dense  subspace $\Val^\infty(W)\subset \Val(W)$ of smooth valuations is by definition the space of smooth vectors of the natural  $GL(W)$-action on $\Val(W)$.  It comes equipped with a natural  Frech\'et space topology and also satisfies McMullen's decomposition 
$$\Val^\infty(W)= \bigoplus_{i=0}^n  \Val_i^\infty(W),$$
where $\Val_k^\infty(W)=\Val_k(W)\cap\Val^\infty(W)$. If $C\subset \RR^n$ is a  convex body of class $C^\infty_+$, i.e., has a smooth boundary with positive Gauss curvature, then $\phi(A)= \vol_n (A+C)$ is a smooth valuation. Consequently, also the  valuations $\phi(A)=  V(A[i],C_1, \ldots, C_{n-i})$, where $C_1,\ldots, C_{n-i}\in\calK(W)$ are of class $C^\infty_+$, are smooth.

\subsubsection{Alesker product} 
The Alesker product is  the continuous, bilinear map $$\Val^\infty(W)\times \Val^\infty(W)\to \Val^\infty(W)$$ 
uniquely characterized by
$$ \vol_W(\Cdot + A) \cdot \vol_W(\Cdot + B) =  \vol_{W\times W} (\Delta(\Cdot) + A\times B),$$
where $A,B\in\calK(W)$ are of class $C^\infty_+$, $\vol_W$ is a Lebesgue measure on $W$, $\vol_{W\times W}$ is the corresponding product measure, and $\Delta\colon W\to W\times W$ denotes the diagonal embedding. Equipped with the product, the space $\Val^\infty(W)$ becomes an associative, commutative graded algebra with the Euler characteristic $\chi$ as multiplicative identity. The Euler-Verdier involution $\sigma\colon \Val^\infty(W)\to \Val^\infty(W)$, defined by 
$$  (\sigma \phi)(A) = (-1)^k \phi(-A)$$
if $\phi$ is homogenous of degree $k$, is an algebra  automorphism for the Alesker product.

\subsubsection{Bernig-Fu convolution}
 
 The space of translation-invariant (or Haar) measures on $W$ is $1$-dimensional and  we denote by $Dens(W)$ its complexification. The elements of $Dens(W)$ are called  densities or Lebesgue measures on $W$. Note that if $\vol_W\in Dens(W)$ is a non-zero Lebesgue measure and $\vol_W^*$ is the dual basis of $Dens(W)^*$, then 
$$ \vol_W(\Cdot+ A) \otimes \vol_W^* \in \Val^\infty(W) \otimes Dens(W)^*$$
does not depend on the choice of $\vol_W$. The Bernig-Fu convolution is the unique continuous, bilinear map
$$\Val^\infty(W) \otimes Dens(W)^*\times \Val^\infty(W) \otimes Dens(W)^*\to \Val^\infty(W) \otimes Dens(W)^*$$
satisfying
$$ \big(\vol_W(\Cdot + A)\otimes \vol_W^*\big) * \big(\vol_W(\Cdot + B)\otimes \vol_W^* \big)=  \vol_{W} (\Cdot + A+ B)\otimes \vol_W^*$$
for any convex bodies $A,B\subset W$ of class $C^\infty_+$.  Equipped with the convolution, $\Val^\infty(V)\otimes Dens(W)^*$ is an associative, commutative graded algebra with the multiplicative identity $\vol_W\otimes \vol_W^*$. We will use repeatedly that for any convex body $C\subset W$ of class $C^\infty_+$ and any smooth valuation $\phi\in\Val^\infty(W)$, 
\begin{equation}\label{eq:convolutionDt} \big(V(\Cdot[n-1],C) * \phi\big)(A) = \frac 1 n \dt \phi(A+t C),\quad A\in\calK(W),\end{equation}
 Moreover,  for any $B_1,\dots,B_k,C_1,\dots,C_l\in\calK(W)$ of class $C^\infty_+$ with $k+l\leq n$, we have
\begin{align}
\label{eq:convolutionMV}
\begin{split}
&V(\Cdot[n-k],B_1,\dots,B_k)*V(\Cdot[n-l],C_1,\dots,C_l)\\
&\qquad=\frac{(n-k)!(n-l)!}{(n-k-l)!n!} V(\Cdot[n-k-l],B_1,\dots,B_k,C_1,\dots,C_l).
\end{split}
\end{align}
Here we  have for simplicity fixed a Lebesgue measure $\vol_W$ which yields the identification $Dens(W)^*\cong\CC$.

\subsubsection{Alesker-Fourier transform}

\label{sss:AFT}

 The two multiplicative structures on valuations introduced above are in a precise sense dual to each other; namely, there is an isomorphism of topological vector spaces
  $$ \FF \colon \Val^\infty(W)\to \Val^\infty(W^*)\otimes Dens(W),$$
 called the Alesker-Fourier transform, which has the following properties:
  \begin{enuma}
   \item $\FF$ commutes with the natural actions of $GL(W)$.
   \item $\FF$ is an isomorphism of graded algebras when the source is equipped with the Alesker product and the target with the Bernig-Fu convolution.
   \item Under the natural identification $Dens(W)\otimes Dens(W^*)\cong \CC$, the Plancherel-type formula holds for $\FF$; namely, for any $A\in\calK(W)$, one has
   $$ \FF^2 \phi(A)=\phi(-A).$$
  \end{enuma}

 Let us fix a Euclidean inner product on $W$ and thus identify $W\cong W^*\cong\RR^n$. Then the Alesker-Fourier transform $\FF \colon \Val^\infty\to \Val^\infty$ satisfies, in particular,
\begin{equation} \label{eq:FFeuler}
 \FF(\chi)= \vol_n \quad \text{ and } \quad \FF(\vol_n)=\chi.
\end{equation}
It is important to note that while  $\FF$ preserves the parity of a valuation, it does not preserve the subspace of real-valued valuations; more precisely,  for any $\phi\in \Val_k^\infty$ we have 
\begin{equation} \label{eq:FFconjuation}
 \FF (\overline{ \phi}) =(-1)^k \overline{\sigma (\FF \phi)},
\end{equation}
 see \cite[Lemma 6.1]{Kotrbaty:HR}.

 In the case of even valuations, the Alesker-Fourier transform admits a simple geometric description. Let $\phi\in\Val_{k}^\infty(W)$ be even and let $E\subset W$  be a $k$-dimensional subspace.  By Hadwiger's characterization of volume, the restriction of $\phi$ to $ \calK(E)$ satisfies
$$ \phi|_E\in Dens(E).$$
 Under the identification $W\cong \RR^n$, we thus get a function $\Klain_\phi \colon \Grass_k(\RR^n)\to \CC$ on the Grassmannian of $k$-planes in $\RR^n$, the Klain function of $\phi$, defined by
$$ \phi|_E= \Klain_\phi(E) \vol_k, \quad  E\in \Grass_k(\RR^n).$$
It is an important result of Klain \cite{Klain:Even} that $\phi$ is uniquely determined by its Klain function. Finally, the Klain function of the Alesker-Fourier transform  of $\phi$ is 
\begin{equation}\label{eq:FFKlain} \Klain_{\FF \phi}(E^\perp) = \Klain_\phi(E),\end{equation}
where $\perp$ denotes the orthogonal complement.

\subsection{Convex geometry}

 Let us now proceed to recall some statements from general convexity theory, in particular about mixed volumes and mixed area measures. For reference and further background on convex geometry we refer the reader to Schneider's book \cite{Schneider:BM}.

\subsubsection{Support function}
 Let $S^{n-1}\subset\RR^n$ be the unit sphere. We denote by $h_A\maps{ S^{n-1}}{\RR}$ the support function of a convex body $A\subset \RR^n$. Sometimes, depending on the context, the same symbol will also denote its $1$-homogeneous extension to $\RR^n$.  We will need that if $E\subset \RR^n$ is a linear subspace and $P_E\colon \RR^n\to E$ the orthogonal projection, then 
\begin{equation}\label{eq:supportProj}h_{P_E A} (x)= h_A(P_E x), \quad x\in \RR^n.\end{equation}

\subsubsection{Mixed area measures}

To any convex body $A\subset\RR^n$, a non-negative measure $S_{n-1}(A)$ on the unit sphere, called the area measure of $A$, is assigned as follows:
 $$S_{n-1}(A)(U)= \vol_{n-1}(\{ x\in \partial A\colon \text{some }u\in U \text{ is an outward normal to } A \text{ at }x\}).$$
 Further, the area measures $S_{i}(A)$ of orders $i=0,\ldots,n-2$ are given by 
 \begin{equation} \label{eq:Steiner}S_{n-1}(A+ rD^n) = \sum_{i=0}^{n-1} \binom{n-1}{i} r^{n-1-i} S_i(A), \quad r\geq0,\end{equation} 
 where $D^n\subset\RR^n$ is the Euclidean unit ball.

Observe that $S_{n-1}$ is homogeneous of degree $n-1$. Accordingly, the mixed area measure $S(A_1,\dots, A_{n-1})$ of  convex bodies $A_1,\dots, A_{n-1}\subset\RR^n$ is defined to be the polarization of $S_{n-1}$. It is non-negative, symmetric under permutations of its arguments,   satisfies 
$S(A,\ldots, A)= S_{n-1}(A)$, and is related to the mixed volume by
\begin{equation}\label{eq:mixedarea}
 V(A_1,\ldots,A_n) = \frac{1}{n} \int_{S^{n-1}} h_{A_n}\, dS(A_1,\ldots, A_{n-1}).
\end{equation}

\subsubsection{Mixed projection body}

 In what follows we will use the notation $V_k$ to emphasize that the mixed volume is taken in a $k$-dimensional  linear subspace. Recall that the mixed projection body $\Pi(A_1\ldots, A_{n-1})$ of  $A_1,\ldots, A_{n-1}\in\calK(\RR^n)$ is defined  in terms of its support function as follows:
\begin{equation}\label{eq:DefMixedProj} h_{\Pi(A_1\ldots, A_{n-1})}(u)= V_{n-1}( P_{u^\perp} A_1, \cdots, P_{u^\perp} A_{n-1}),\quad u\in S^{n-1},\end{equation}
where $P_{u^\perp } \colon \RR^n\to u^\perp$ is the orthogonal projection.  Observe that since by \eqref{eq:mixedarea}
\begin{align*}   V_{n-1}( P_{u^\perp} A_1, \cdots, P_{u^\perp} A_{n-1})&  =  \frac n 2  V( [-u,u], A_1,\ldots, A_{n-1}) \\
 & =\frac12 \int_{S^{n-1}} |\langle u,v\rangle| \,dS(A_1,\ldots A_{n-1}),
\end{align*} 
\eqref{eq:DefMixedProj} indeed defines a support function.

\subsubsection{Blaschke addition}

Related to the notion of area measure is  further a binary operation on the subset of full-dimensional convex bodies  that is defined as follows. According to  Minkowski's existence theorem~\cite[Thm~8.2.1]{Schneider:BM}, for any  full-dimensional convex bodies $A,B\in\calK(W)$, there exists $C\in\calK(W)$ with
\begin{align*}
S_{n-1}(C)=S_{n-1}(A)+S_{n-1}(B),
\end{align*}
after some choice of identification $W\cong \RR^n$.
Such $C$, which is moreover unique up to translations, is denoted by $A\# B$ and called the Blaschke sum of $A$ and $B$.

\subsubsection{Mixed discriminants}

 For convex bodies of class $C^\infty_+$, the mixed area measure, and hence the mixed volume, can be described in terms of the mixed discriminant of the Hessians of their support functions.

Let $g_{S^{n-1}}$ denote the Riemannian metric on $S^{n-1}$ induced from $\RR^n$ and let  $\nabla$ be  its Levi-Civita connection.  For each smooth function $f$ on the sphere,   $\nabla^2 f + f g_{S^{n-1}}$ is  a symmetric $2$-tensor field. At each point of the sphere, this gives a symmetric bilinear form on the tangent space  and we may consider its determinant. The area measure of a  convex body $A\subset\RR^n$ of class $C^\infty_+$ is then simply
$$  d S_{n-1}(A)=\det\big(\nabla^2 h_A + h_A g_{S^{n-1}}\big) \,du,$$
where $du$ is the spherical Lebesgue measure.  Further, we conclude from \eqref{eq:Steiner} and $\tr (\nabla^2 f + f g_{S^{n-1}})= \Delta_{S^{n-1}} f + (n-1) f$, where $\Delta_{S^{n-1}}$ is the Laplace-Beltrami operator,  that

\begin{equation}\label{eq:area1}  dS_1(A) = \left(\frac{1}{n-1}\Delta_{S^{n-1}} h_A+  h_A \right)du. \end{equation}
Finally, for any smooth functions $f_1,\ldots, f_{n-1}$ on $S^{n-1}$ and each point $u\in S^{n-1}$  we define
$ D(f_1,\ldots, f_{n-1})(u)$  to be the  mixed discriminant of the  following symmetric bilinear forms:
$$\left.\nabla^2 f_i +  f_i g_{S^{n-1}}\right|_u,\quad  i=1,\ldots, n-1.$$
With this terminology, for any $A_1,\dots,A_{n-1}\in\calK(\RR^n)$ of class $C^\infty_+$ we have
{
\begin{equation}\label{eq:mixedareaD} dS(A_1,\ldots, A_{n-1})=D(h_{A_1},\ldots, h_{A_{n-1}})\, du.\end{equation}

\section{Proof of Theorem~\ref{thm:MixedConvolution}}

  First of all, recall that since the mixed volume, in each argument, is Minkowski additive it extends naturally to a linear functional on differences of support functions. In particular,  since every smooth function on the sphere can be expressed as the difference of two support functions, the mixed volume is well defined on $C^\infty(S^{n-1})$, see \cite[Section 5.2]{Schneider:BM}. 
By $\CC$-linear extension we may just as well assume the smooth functions  to take values in the complex numbers. With this terminology,  Aleksandrov  deduced the  Aleksandrov-Fenchel inequality from the following theorem.

\begin{theorem}[Aleksandrov  \cite{Aleksandrov:Theorie4}]\label{thm:elliptic}
Let $C_1,\ldots, C_{n-1} \subset \RR^n$ be convex bodies of class $C^\infty_+$ and let $f\colon S^{n-1}\to \CC$ be a smooth function.
 \begin{enuma}
\item If $ D(f, h_{C_1},\ldots, h_{C_{n-2}})=0$, then $f$ is the restriction of a linear functional on $\RR^n$.
\item Consider $D\colon C^\infty(S^{n-1})\to C^\infty(S^{n-1})$ given by $ Df = D(f, h_{C_1},\ldots, h_{C_{n-2}})$. $D$ is an elliptic operator, self-adjoint with respect to the standard $L^2$-inner product. Consequently, $\im D = (\ker D)^\perp$. 
\item If $V(f,C_1,\ldots, C_{n-1})=0$, then  $V( \overline f,f,C_1,\ldots, C_{n-2})\leq 0$ and equality holds if and only if $f$ is the restriction of a linear functional on $\RR^n$.
 \end{enuma}
\end{theorem}

\begin{remark}
An English translation of the article \cite{Aleksandrov:Theorie4} can be found in Aleksandrov's collected works  \cite{Aleksandrov:Collected}; for a modern exposition see also  H\"ormander's book \cite{Hormander:Convexity}. The equality  $\im D = (\ker D)^\perp$ is not explicitly stated (since it is not used), but this is just a fundamental property of elliptic operators, see, e.g., \cite[Theorem III.5.5]{LawsonMichelson:Spin}.
\end{remark}

\begin{lemma}
\label{lemma:isom}Let  $\calH\subset C^\infty(S^{n-1})$ be the subspace of restrictions of  linear functionals on $\RR^n$.  Consider $F\maps{ \calH^\perp}{\Val^\infty_{n-1}}$ and $G\maps{ \calH^\perp}{\Val^\infty_{1}}$ defined by 
$$ Ff(A)= \int_{S^{n-1}} f\, dS_{n-1}(A),\quad A\in\calK(\RR^n)$$
and $$ Gf(A)= \int_{S^{n-1}} h_A(u) f(u) \,du,\quad A\in\calK(\RR^n),$$respectively. Then $F$ and $G$ are isomorphisms of topological vector spaces. 
 
\end{lemma}

\begin{proof} For $F$, the assertion is well known,   see, e.g.,  \cite[Theorem 4.1]{SchusterW:Generalized} or \cite[Theorem A.2]{BPSW:Log} for a proof.

As for the operator $G$, observe that $\calH$ and $\calH^\perp$ are precisely the kernel and the image of the elliptic operator  $\Delta_{S^{n-1}} + (n-1) I$, and thus to each $f\in \calH^\perp$ there exists $g\in \calH^\perp$ with
$$ \frac{1}{n-1}\Delta_{S^{n-1}} g    +g= f.$$ 
Using  that the Laplace-Beltrami operator is self-adjoint  together with \eqref{eq:area1}, for any convex body $A$ of class $C^\infty_+$ we have 
$$ Gf(A) 
  = \int_{S^{n-1}} \left(\frac{1}{n-1}\Delta_{S^{n-1}}h_A+  h_A\right) g \,du = \int_{S^{n-1}} g \,dS_1(A).
$$ The claim follows now from \cite[Theorem 4.1]{SchusterW:Generalized} or \cite[Theorem A.2]{BPSW:Log}.
 \end{proof}

 \begin{remark} Observe that, according to \eqref{eq:mixedarea}, an alternative expression for $F$ is
\begin{align*}
Ff=nV(\Cdot[n-1],f).
\end{align*}

\end{remark}

\begin{proof}[Proof of Theorem~\ref{thm:MixedConvolution}]
Consider the operator $D\maps{ C^\infty(S^{n-1})}{C^\infty(S^{n-1})}$ given by $ Df = D(f, h_{C_1},\ldots, h_{C_{n-2}})$. According to Theorem~\ref{thm:elliptic},  $D$ is self-adjoint,  its kernel is the subspace of restrictions of linear functionals on $\RR^n$, and $\im D = (\ker D)^\perp$. Let further $L\maps{\Val_1^\infty}{ \Val^\infty_{n-1}}$ be  defined by \eqref{eq:LConvolution} and consider the following diagram:
\begin{center}
\begin{tikzpicture}
  \matrix (m) [matrix of math nodes,row sep=3em,column sep=4em,minimum width=2em]
  {
     \im D  &  \im D \\
     \Val_{n-1}^\infty & \Val_{1}^\infty \\};
  \path[-stealth]
    (m-1-1) edge node [right] {$F$} (m-2-1)
            edge node [above] {$D$} (m-1-2)
    (m-1-2) edge node [right] {$G$} (m-2-2)
     (m-2-1) edge node [above] {$L$}(m-2-2);
\end{tikzpicture}
\end{center}
 Since by Theorem~\ref{thm:elliptic} (a) and Lemma~\ref{lemma:isom}, $F$, $G$, and $D$ are isomorphisms,  the Hard Lefschetz theorem follows once we show that the diagram commutes (possibly up to a normalizing constant).

To this end, using \eqref{eq:convolutionMV}, \eqref{eq:mixedarea}, \eqref{eq:mixedareaD}, and the self-adjointness of $D$, for any convex body $A\subset\RR^n$ of class $C^\infty_+$ we indeed have
\begin{align*}
\big((L\circ F) f\big)(A) &=c_nV(f,A,C_1,\dots,C_{n-2})\\
&  = c_n \int_{S^{n-1}}  f \,dS(A,C_1,\ldots, C_{n-2})\\
 &  = c_n \int_{S^{n-1}}  f D(h_A) \,du\\
 & = c_n \int_{S^{n-1}}  h_A Df \,du\\
 & = c_n \big((G\circ D) f\big)(A),
\end{align*}
where $c_n$ is a non-zero constant which may possibly change its value from line to line but always depends entirely on the dimension.

To prove the Hodge-Riemann relations, recall first from Lemma \ref{lemma:isom} that each $\phi\in\Val_{n-1}^\infty$ can be written uniquely as $\phi=Ff$ for some $f\in\calH^\perp$. Observe further that, according to \eqref{eq:convolutionMV},
$$ \phi * \psi_1*\cdots *\psi_{n-1}=\tilde c_nV(f,C_1,\ldots, C_{n-1})$$
 and similarly 
$$  \overline{\phi} * \phi * \psi_1*\cdots *\psi_{n-2}=\tilde c_n V(\overline f, f, C_1,\ldots, C_{n-2}) ,$$ with some positive real constant $\tilde c_n$ depending only on the dimension. Now the claim follows at once from Theorem~\ref{thm:elliptic} (c).

\end{proof}

\begin{remark}\label{rem:Metricfree1}
 To obtain a formulation of Theorem~\ref{thm:MixedConvolution} that is independent of a choice of Euclidean inner product and  Lebesgue measure, one has to make the following adjustments. First,  the Lefschetz operators should be defined through
$$\psi_i(A)= V(A[n-1], C_i)\otimes \vol_W^*,\quad i=1,\dots,n-1,$$
where the mixed volume $V$ is normalized by a 
non-zero Lebesgue measure $ \vol_W$ on $W$. Note that $\psi_i$ does not depend on the choice of Lebesgue measure. Second, a valuation $\phi\in \Val(W)\otimes Dens(W^*)$ is defined to be non-positive if
$\phi(A,B)\leq 0$ for all convex bodies $A\subset W$ and $B\subset W^*$. 

\end{remark}

\section{Proof of Theorem~\ref{thm:MixedProduct}}

\label{s:MixedProduct}

\begin{lemma}
 If $A_1, \ldots,A_{n-1}\subset \RR^n$ are convex bodies of class $C^\infty_+$, then 
 so is the mixed projection body $\Pi(A_1,\ldots,A_{n-1})$.
\end{lemma}

\begin{proof}
  It suffices to show that the support function $h= h_{\Pi(A_1,\ldots,A_{n-1})}$ is smooth and that the Hessian  
 $\nabla^2 h+ hg_{S^{n-1}}$ is positive definite at every point of the sphere. First, since the cosine transform
 $ T \colon C^\infty(S^{n-1}) \to C^\infty(S^{n-1})$, 
 $$ Tf (u)= \int_{S^{n-1}} |\langle u,v\rangle | f(v) dv
$$
maps smooth functions to smooth ones and since by \eqref{eq:mixedareaD} the mixed area measure has a smooth density with respect to the spherical Lebesgue measure, it is clear that the support function  is smooth.

 Second, since $A_1, \ldots, A_{n-1}$ have smooth boundaries with positive Gauss curvature, there exist convex bodies $B_1,\dots, B_{n-1}$ with the same property and a Euclidean ball $\varepsilon D^n$ such that 
$A_i= B_i+\varepsilon D^n$ for $i=1,\ldots, n-1$. Since the mixed projection body is Minkowski additive in each argument, we have 
$$ \Pi(A_1,\ldots, A_{n-1}) = \Pi(B_1+\varepsilon D^n,\ldots, B_{n-1}+\varepsilon D^n) = B+ \widetilde\varepsilon D^n,$$
where $B$ is a convex body with smooth support function and $\widetilde\varepsilon D^n$ is another Euclidean ball. This proves that $\nabla^2 h+ hg_{S^{n-1}}$ is positive definite.
\end{proof}

\begin{lemma}
\label{lem:fourierUpsilon}
Let $C_1,\dots,C_{n-1}\subset \RR^n$ be centrally symmetric convex bodies of class $C^\infty_+$. The valuations
$$\psi(A)= V(A, C_1,\ldots, C_{n-1})$$
and
$$\upsilon(A)=\frac12V\big(A[n-1],\Pi(C_1,\dots,C_{n-1})\big)$$
satisfy
 \begin{equation}
 \label{eq:fourierUpsilon} \psi= \FF \upsilon.
 \end{equation}
\end{lemma}

\begin{proof}
Observe first that $\psi$ and $\upsilon$ are even and of complementary degree. Fix $u\in S^{n-1}$ and let $B\subset u^\perp$ be a  convex body with $\vol_{n-1}(B)=1$. Let $L_u$ denote the line spanned by $u$.  By \cite[Theorem 5.3.1]{Schneider:BM} and  \eqref{eq:supportProj} we have
 \begin{align*}
  \Klain_{\upsilon}( u^\perp) & = \frac 12 V\big(B[n-1], \Pi( C_{1},\ldots, C_{n-1})\big)\\
  & = \frac{ 1}{2n}  \vol_1 \big( P_{L_u}  \Pi( C_{1},\ldots, C_{n-1}) \big)\\
  & = \frac 1n h_{\Pi( C_{1},\ldots, C_{n-1})} (u)\\
  & =  \frac 1nV_{n-1}( P_{u^\perp} C_{1}, \ldots, P_{u^\perp} C_{n-1})\\
  & =  V([0,u],  C_{1}, \ldots,C_{n-1})\\
  & = \Klain_{ \psi}(L_u)
 \end{align*}
 and the claim thus follows from \eqref{eq:FFKlain}.
\end{proof}

\begin{proof}[Proof of Theorem~\ref{thm:MixedProduct}]  For $i=1,\ldots,n-1$, consider the valuations 
 $$\upsilon_i(A) = \frac 12 V \big(A[n-1], \Pi( C_{i,1},\ldots, C_{i,n-1})\big). $$
 Consider further the operators $L\colon \Val_{n-1}^\infty\to \Val_1^\infty$  and $\Lambda \colon \Val_1^\infty\to \Val_{n-1}^\infty$ defined by $L \phi  = \phi*  \upsilon_1 *\cdots *  \upsilon_{n-2}$ and 
$\Lambda\phi =\phi \cdot  \psi_1 \cdots \psi_{n-2}$. By \eqref{eq:fourierUpsilon} and the fact that the Fourier transform takes the product to the convolution,  the diagram
\begin{center}
\begin{tikzpicture}
  \matrix (m) [matrix of math nodes,row sep=3em,column sep=4em,minimum width=2em]
  {
     \Val_{n-1}^\infty & \Val_{1}^\infty \\
     \Val_{1}^\infty & \Val_{n-1}^\infty \\};
  \path[-stealth]
    (m-1-1) edge node [right] {$\FF$} (m-2-1)
            edge node [above] {$L$} (m-1-2)
    (m-1-2) edge node [right] {$\FF$} (m-2-2)
     (m-2-1) edge node [above] {$\Lambda$}(m-2-2);
\end{tikzpicture}
\end{center}
commutes. Since $\FF$ and $L$ are isomorphisms, so must be $\Lambda$, which proves (a).

 To prove (b), note that using  \eqref{eq:fourierUpsilon}, \eqref{eq:FFeuler}, and \eqref{eq:FFconjuation}  we have 
$$ \phi *  \upsilon_1 *\cdots *  \upsilon_{n-1}=0\ \Longleftrightarrow\ (\FF\phi) \cdot  \psi_1 \cdots \psi_{n-1} = 0$$
and
$$ \overline\phi *  \phi *  \upsilon_1 *\cdots *  \upsilon_{n-2}<0\  \Longleftrightarrow \ \overline{\sigma (\FF \phi)} \cdot (\FF\phi) \cdot  \psi_1 \cdots \psi_{n-2} > 0,$$
and that the assertion thus follows from item (b) of Theorem~\ref{thm:MixedConvolution}.
\end{proof}

\begin{remark}\label{rem:Metricfree2}
 The Euclidean structure enters the formulation of Theorem~\ref{thm:MixedProduct} only at the definition of the Lefschetz operators. To remove this dependence, it is enough to allow the Lefschetz operators to be defined by any positive Lebesgue measure on $W$. 
\end{remark}

\section{Properties of the  higher rank mixed volume}

\label{sec:higherMixed}

 Let us first clarify how the Lebesgue measure is fixed on $\coker \Delta_l$ in Definition~\ref{def:higherMixed}. Let $W_i$, $i=1,2,3$, be finite-dimensional real vector spaces. If 
$$ 0\longrightarrow W_1\longrightarrow W_2\longrightarrow W_3\longrightarrow 0$$
is an exact sequence, there exists a canonical isomorphism 
\begin{equation}\label{eq:densities} Dens (W_2) \otimes Dens(W_1)^* \cong   Dens(W_3),\end{equation}
 see, e.g., \cite[Section 2.1]{Alesker:Fourier} for details. 
 
 Let us apply this  statement to the situation 
 $$ 0 \longrightarrow \im \Delta_l\longrightarrow W^l \longrightarrow \coker \Delta_l \longrightarrow 0,$$
where $\Delta_l\maps{W}{W^l}$ is the diagonal embedding. Any Lebesgue measure $\vol_W$ on $W$ canonically induces, first, a Lebesgue measure $\vol_{\im \Delta_l}$ on $\im \Delta_l$ (since $\Delta_l\maps W{\im \Delta_l}$ is an isomorphism), and second, a Lebesgue measure $\vol_{W^l}$ on $W^l$ (the product measure).  The Lebesgue measure $\vol_{\coker \Delta_l}$ on $\coker \Delta_l$ is then defined as the image of $ \vol_{W^l}\otimes \vol_{\im \Delta_l}^*$ under the isomorphism \eqref{eq:densities}.
 
 The following lemma will be important.
 
\begin{lemma}[Alesker {\cite[Lemma 2.6.1]{Alesker:Fourier}}] \label{lem:volumeProjection} Let $ 0\longrightarrow W_1\stackrel{f}{\longrightarrow} W_2\stackrel{g}{\longrightarrow} W_3\longrightarrow 0$ be exact and let $\vol_{W_1}$ and $\vol_{W_2}$ be Lebesgue measures on $W_1$ and $W_2$.  Let $\vol_{W_3}$ be the image of $\vol_{W_2}\otimes \vol_{W_1}^*$ under the isomorphism \eqref{eq:densities}. Put $n=\dim W_2$ and $k=\dim W_1$. Then for any convex bodies $A\subset W_2$, $B\subset W_1$ we have
$$ \binom{n}{k} V\big(A[n-k], f(B)[k]\big) = \vol_{W_1}(B) \vol_{W_3}\big(g(A)\big),$$
where the mixed volume on the left-hand side is normalized by $\vol_{W_2}$.
\end{lemma}

For the rest of this section, $W$ will be an $n$-dimensional vector space and $\vol_W$ a chosen positive Lebesgue measure on it. Similarly, we will keep also the rest of the notation of Definition \ref{def:higherMixed}.

\begin{proposition} \label{prop:higherMixed}
 The higher rank mixed volume defined by \eqref{eq:higherMixed} has the following properties:
\begin{enuma}
 \item $\widetilde V(\calA_{\rho(1)},\dots,\calA_{\rho(l)})=\widetilde V(\calA_1,\dots,\calA_l)$ for any  permutation $\rho\in\fSS_l$.
 \item $\widetilde V(g\calA_1,\ldots, g\calA_n) = \widetilde V(\calA_1,\ldots, \calA_n)$, where $g\calA_i=(gA_{i,1},\dots,gA_{i,n-k_i})$,  for any $g\in SL(W)$.
  \item For fixed $\calA_2,\ldots, \calA_n$, the function  $A\mapsto \widetilde V\big((A[n-k_1]),\calA_2,\ldots,\calA_{l}\big) $ is  a $(n-k_1)$-homogeneous translation-invariant continuous valuation.
 \item $\widetilde V\geq 0$ and $\widetilde V$ is monotonically increasing in each argument.
 \item $\widetilde V(\calA_1,\ldots, \calA_l)>0$ if and only if there exist line segments $S_{i,j} \subset A_{i,j}$ such that 
 \begin{enumi}
  \item $\dim H_i =n-k_i$ for $i=1,\ldots,l$ and 
  \item $\bigcap_{i=1}^l H_i=\{0\}$,
 \end{enumi}
  where  $H_i\subset W$ is the translate of the  affine  hull of $S_{i,1},\ldots, S_{i,n-k_i}$ containing the origin.

\end{enuma}
\end{proposition}

\begin{proof}  
Items (c) and (d) follow at once from the respective properties of the mixed volume.

To show (a), fix a permutation $\rho\in \fSS_l$ and consider $\rho\colon W^l\to W^l$ defined by $$\rho(w_1,\ldots,w_l) = (w_{\rho(1)},\ldots,w_{\rho(l)}).$$ Since $\rho$ leaves invariant the subspace  $\im\Delta_l\subset W^l$, it descends to  the unique isomorphism $\widetilde \rho \colon\coker \Delta_l \to \coker \Delta_l $ satisfying
$\pi \circ \rho = \widetilde \rho \circ\pi$ for the canonical projection $\pi$. Since $\rho$ is volume-preserving and $\widetilde \rho \circ f_i=f_{\rho(i)}$, the symmetry of $\widetilde V$ follows. 

The proof of (b) is similar, using that $g\times \cdots \times g\maps{W^l}{W^l}$ descends to a map $\widetilde g\colon \coker \Delta_l\to\coker \Delta_l$.  

To prove (e) observe that by \cite[Theorem 5.1.8]{Schneider:BM} we have $\widetilde V(\calA_1,\ldots, \calA_l)>0$ if and only if there exist line segments $S_{i,j} \subset A_{i,j}$ such that $f_i(S_{i,j})$, $i=1,\ldots l$ and $j=1,\ldots, n-k_i$, are line segments with linearly independent directions. The latter is the case if and only if $\dim H_i=n-k_i$ and 
the intersection of $\iota_1(H_1)+ \cdots + \iota_l(H_l)$ with the subspace $\im \Delta_l$ consists precisely of the origin. Since 
$$ \big( \iota_1(H_1)+ \cdots + \iota_l(H_l)\big) \cap \im \Delta_l= \Delta_l\left( \bigcap_{i=1}^l H_i\right),$$
the claim follows.
\end{proof}

\begin{lemma}
\label{lem:coeff}
Let $k_1,\dots,k_m$ be non-negative integers with $k=\sum_{i=1}^mk_i\leq n$. For any $A,C_1,\dots,C_m\in\calK(W)$ we have
\begin{align*}
&\left.\frac{\partial^k}{\partial t_1^{k_1}\cdots\partial t_m^{k_m}}\right|_{t_1=\cdots=t_m=0}\vol_W(A+t_1C_1+\cdots+t_mC_m)\\
&\qquad=\frac{n!}{(n-k)!}V(A[n-k],C_1[k_1],\dots,C_m[k_m]).
\end{align*}
\end{lemma}

\begin{proof} Since the left and right hand side must clearly be proportional, choosing $C_1=\cdots=C_m=A$ yields the  constant.
\end{proof}

\begin{proposition}
\label{prop:higherMixed2}

Let $n=k_1+\cdots+k_l$ be a partition into positive integers. For $i=1,\dots, l$, let $\calC_i=(C_{i,1},\dots,C_{i,n-k_i})$ be an $(n-k_i)$-tuple of convex bodies of class $C^\infty_+$ and consider the valuation
 $$\phi_i(A)=V(A[k_i],C_{i,1},\dots,C_{i,n-k_i}).$$ 
 Then 
 $$ \phi_1\cdots \phi_l=    \frac{k_1!\cdots k_l!(n(l-1))!}{(n!)^l}    \widetilde  V(  \calC_1,  \cdots ,  \calC_l) \vol_W .$$

\end{proposition}
\begin{proof}
Let us first consider the special case when $\calC_i=(C_i[n-k_i])$, $i=1,\dots, l$. For any $C\in\calK(W)$ of class $C^\infty_+$ put $\psi_{C}=\vol_W(\Cdot + C)$. By \cite[Lemma 3.1]{Alesker:Kotrbaty} we have 
$$\psi_{C_1}\cdots \psi_{C_l}
 = \vol_{W^l}\big( \Delta_l(\Cdot) + C_1\times \cdots \times C_l\big).$$
  For $j=1,\dots,l$ let further $\iota_j\colon W\hookrightarrow  W^l$ be the inclusion into the $j$-th summand. Since 
 $$ \phi_j = \frac{k_j!}{n!}\left.\frac{d^{n-k_j}}{dt^{n-k_j}}\right|_{t=0} \psi_{tC_j},$$
according to Lemma~\ref{lem:coeff} and Lemma~\ref{lem:volumeProjection} for any $A\in\calK(W)$ we have
 \begin{align*} &  (\phi_1\cdots \phi_l)(A)  =\\
 &=\frac{k_1!\cdots k_l!}{(n!)^{l}} \left.\frac{\partial^{n(l-1)}}{\partial t_1^{n-k_1} \cdots \partial t_l^{n-k_l}}\right|_{t_1=\cdots=t_l=0} \vol_{W^l}\big(\Delta_l(A) + t_1 \iota_1(C_1) + \cdots + t_l \iota_l(C_l)  \big)\\
 &  = \frac{k_1!\cdots k_l!(nl)!}{(n!)^{l+1}}  V_{ W^l}\big(\Delta_l(A)[n], \iota_1(C_1)[n-k_1],\ldots, \iota_l(C_l)[n-k_l]\big)\\
 & =  \frac{k_1!\cdots k_l!(nl)!}{(n!)^{l+1}} \binom{nl}{n}^{-1}  \vol_W(A) V_{\coker\Delta_l}\big( f_1(C_1)[n-k_1],\ldots, f_l(C_l)[n-k_l]\big)\\
 & = \frac{k_1!\cdots k_l!(n(l-1))!}{(n!)^{l}}    \vol_W(A)\widetilde  V(\calC_1,\ldots, \calC_l).
 \end{align*}
The general case now follows by polarization, i.e.,  by replacing $C_i$ by $\sum_{j=1}^{n-k_i}  t_{i,j} C_{i,j}$, expanding both sides into a polynomial in $t_{i,j}\geq 0$, and comparing their coefficients. 

\end{proof}

Finally, let us investigate the relation between $\widetilde V$ and the usual mixed volume. First, we show that  the mixed volume of rank $1$ yields the usual mixed volume.

\begin{proposition}\label{prop:higherMixedUsual}

  Let  $A_1\ldots, A_{n-k},B_1,\dots,B_{k}\subset W$ be convex bodies and write $\calA=(A_1,\ldots, A_{n-k})$ and $\calB=(B_1,\dots,B_{k})$. Then
  \begin{align*}
  \widetilde V(\calA,\calB)=V(A_1,\ldots,A_{n-k},-B_1,\dots, -B_{k}).
  \end{align*}

\end{proposition}

\begin{proof} Without loss of generality we may assume that the convex bodies $A_1\ldots, A_{n-k}$ and  $B_1,\dots,B_{k}$ are of class $C^\infty_+$.
 Consider the valuations
\begin{align*}
\phi_1(A)=V(A[k], A_1,\ldots, A_{n-k})
\end{align*}
 and
 \begin{align*}
 \phi_2(A)=V(A[n-k], B_1,\ldots, B_{n-k}).
 \end{align*} 
According to \cite{Alesker:Product}*{Proposition 2.2},
$$\phi_1\cdot \phi_2= \binom{n}{k}^{-1} V(A_1,\ldots,A_{n-k},-B_1,\dots, -B_{k}) \vol_W.$$
At the same time, by Proposition~\ref{prop:higherMixed2},
$$\phi_1\cdot \phi_2= \binom{n}{k}^{-1}\widetilde V(\calA,\calB)  \vol_W$$
and the claim follows.
\end{proof}

Second, we show that the 
mixed volume of rank   $n-1$, which we will exclusively consider from now on, boils down to the usual mixed volume in a particular situation.

\begin{proposition}\label{prop:higherMixedSpecial} Let $\calC_1,\dots,\calC_{n}$ be $(n-1)$-tuples of convex bodies in $\RR^n$. If the bodies in $\calC_1,\ldots, \calC_{n-1}$ are centrally symmetric, then 
   $$ \widetilde V(\calC_1,\ldots,\calC_n) =  \frac{ n!((n-1)!)^n}{2^n(n(n-1))!}   V(\Pi \calC_1,\ldots, \Pi \calC_n).$$

\end{proposition}

\begin{proof}

Similarly as in the proof of Proposition~\ref{prop:higherMixed2} it causes no loss of generality to assume $\calC_i=(C_i[n-1])$, $i=1,\dots,n$. Moreover, we may clearly assume that the convex bodies are all of class $C^\infty_+$.  For $i=1,\dots,n$ consider the valuations
\begin{align*}
\psi_i(A)=V(A, C_i[n-1])
\end{align*}
and
\begin{align*}
\upsilon_i(A) = \frac 12 V (A[n-1], \Pi \calC_i).
\end{align*}

Let us first assume that also $C_n$ is centrally symmetric.  In this case, by Proposition~\ref{prop:higherMixed2}, Lemma~\ref{lem:fourierUpsilon}, and  \eqref{eq:convolutionMV} we have
\begin{align*} 
\frac{(n(n-1))!}{(n!)^n}     \widetilde  V(  \calC_1,  \cdots ,  \calC_n) \vol_n  & =\psi_1\cdots \psi_n \\
&=\FF( \upsilon_1*\cdots * \upsilon_n )\\
& =  \frac{ ((n-1)!)^n}{2^n (n!)^{n-1}} V(\Pi \calC_1,\ldots, \Pi \calC_n) \vol_n.
\end{align*} 

In general, we decompose $\psi_n= \psi^+ + \psi^-$ into its even and odd part. Since there are no  non-trivial odd valuations in degree $n$, we have $ \psi_1\cdots \psi_{n-1} \cdot \psi^-=0$ and hence
$$ \psi_1\cdots \psi_{n-1}\cdot  \psi_n=  \psi_1\cdots \psi_{n-1} \cdot \psi^+.$$
To conclude the proof it suffices to observe  that
\begin{align*}
\psi^+(A)=\frac12 V\big( A, C_n \# (-C_n) [n-1]\big)
\end{align*}
 and

\begin{align*}
\Pi\big(C_n \# (-C_n) [n-1]\big)= \Pi(C_n[n-1])+\Pi(-C_n[n-1])= 2\Pi(C_n[n-1]).
\end{align*}
\end{proof}

\section{Geometric consequences of the Hodge-Riemann relations}

We apply our results to the study of inequalities between mixed volumes. Our first result along with its proof is a direct generalization of \cite[Theorem 1.1, item (1)]{Alesker:Kotrbaty}.

\begin{theorem}
\label{thm:DMVminus}
Let $W$ be an $n$-dimensional real vector space with a fixed positive Lebesgue measure. Let $\calA,\calC_1,\dots,\calC_{n-2}$ be $(n-1)$-tuples of convex bodies in $W$ and assume that the bodies in $\calC_i$, $i=1,\dots,n-2$, are centrally symmetric. Then
\begin{align*}
\widetilde V(\calA,-\calA,\calC_1,\dots,\calC_{n-2})\leq  \widetilde V(\calA,\calA,\calC_1,\dots,\calC_{n-2}).
\end{align*}
\end{theorem}

\begin{proof}
Write $\calA=(A_1,\dots,A_{n-1})$ and $\calC_i=(C_{i,1},\dots,C_{i,n-1})$, $i=1,\dots,n-2$. By continuity we may assume that all these bodies are of class $C_+^\infty$. Consider now the valuations given by
$$\phi(A)=V(A,A_{1},\dots,A_{n-1})-V(A,-A_{1},\dots,-A_{n-1})$$
and
$$\psi_i(A)=V(A,C_{i,1},\dots,C_{i,n-1}),\quad i=1,\dots,n-2.$$
Since $\phi$ is 1-homogeneous and odd, it is necessarily primitive and satisfies $\sigma\phi=\phi$. Consequently, the Hodge-Riemann relations (Theorem~\ref{thm:MixedProduct}) imply
$$\phi^2\cdot\psi_1\cdots\psi_{n-2}\geq 0.$$
Expanding the product according to  Proposition~\ref{prop:higherMixed2}, the claim follows.
\end{proof}

The Hodge-Riemann relations (Theorem~\ref{thm:MixedProduct})  yield an analog of the Aleksandrov-Fenchel inequality for $\widetilde V$ as follows.

\begin{proof}[Proof of Theorem~\ref{thm:DAF}]
Since   $\widetilde V$ is continuous, we may assume that the convex bodies $A_{i,j}$ and $C_{i,j} $  are all of class $C^\infty_+$. Put
$$\phi_i(A) = V(A,A_{i,1},\ldots, A_{i,n-1}),\quad i=1,\ldots,m,$$
and 
$$ \psi_i(A)= V(A,C_{i,1},\ldots, C_{i,n-1}),\quad i=1,\ldots,n-2.$$
Observe that $ Q(\mu,\nu)= \mu\cdot\overline{\sigma \nu}  \cdot \psi_1\cdots\psi_{n-2}$ defines a symmetric sesquilinear form on $\Val_1^\infty$ and that, according to Proposition \ref{prop:higherMixed2},
\begin{align*}
v_{ij}=v_{ji}=-\frac1{c_n}Q(\phi_i,\phi_j),
\end{align*}
for $i,j=1,\dots,m$ and some positive constant $c_n$ depending only on the dimension.

Observe further that $v_{mm}>0$ and consider the 1-homogeneous smooth valuations
$$ \mu_i = \phi_i - \frac{v_{im}}{v_{mm}} \phi_m, \qquad i=1,\ldots, m-1.$$
Each $\mu_i$ is primitive in the sense that
\begin{align*}
\mu_i\cdot\b{\sigma\phi_m}\cdot\psi_1\cdots\psi_{n-2}&=-c_n\left(v_{im}-\frac{v_{im}v_{mm}}{v_{mm}}\right)=0.
\end{align*}
Consequently, the Hodge-Riemann relations (Theorem~\ref{thm:MixedProduct}) imply that the $(m-1)\times(m-1)$ matrix $(a_{ij})$ with
\begin{align*}
a_{ij}=Q(\mu_i,\mu_j)=-c_n\left(v_{ij}-\frac{v_{im}v_{jm}}{v_{mm}}\right)
\end{align*}
is positive semi-definite and hence $\det(a_{ij})\geq0$. Subtraction of rows easily yields
\begin{align*}
\det(v_{ij})=\det\begin{pmatrix}-c_n(a_{ij})&0\\ * &v_{mm}\end{pmatrix},
\end{align*} 
where `$*$' stands for an insignificant part of the matrix. Therefore we indeed have
\begin{align*}
(-1)^m\det(v_{ij})=-(c_n)^{m-1}\det(a_{ij})v_{mm}\leq0.
\end{align*}
\end{proof}

 We finally turn to the proof of Theorem~\ref{thm:AFfunctionals}, the generalization of Corollary~\ref{cor:R3}.
Let us first prove a technical lemma, assuming that, as usual, $W$ is an $n$-dimensional real vector space and $\vol_W$ is a choice of positive Lebesgue measure on it, and keeping the notation from Definition \ref{def:higherMixed}.

\begin{lemma}
\label{pro:MVfunctionals}
Let $1\leq k\leq n-1$ and let $\xi_1,\ldots, \xi_k\in W^*$ be linearly independent linear functionals. Consider the  map $\Delta_{n-k-1}\times L_k\maps{W}{ W^{n-k-1}\times \RR^k}$, where $\Delta_{n-k-1}$ is the diagonal embedding and $L_k=(\xi_1,\dots,\xi_k)$. Take any convex bodies $A_1, \ldots, A_{n-k}, C_1,\ldots, C_k\subset W$ with $C_i\subset \ker \xi_i$ and $\dim C_i=n-1$ for $i=1,\ldots, k$,  and write $\calA_j=(A_j[n-1])$, $j=1,\dots,n-k$, and $\calC_i=(C_i[n-1])$, $i=1,\dots, k$. Then
\begin{align*}
&\widetilde V(\calA_1,\ldots, \calA_{n-k}, \calC_1,\ldots, \calC_k)\\
&\quad=  c V\big(\iota_1(A_1)[n-1],\ldots, \iota_{n-k-1}(A_{n-k-1})[n-1], (\Delta_{n-k-1}\times L_k)(-A_{n-k})[n-1]\big),
\end{align*}
where the mixed volume is taken in $W^{n-k-1}\times\RR^k$ and $c$ is a  positive real constant that  does not depend on  $A_1,\ldots,A_{n-k}$.
\end{lemma}

\begin{proof}
We will proceed by induction. To this end, assume first that $k=1$. Clearly the linear map $f\colon W^n\to W^{n-1}$ given by
$$f(w_1,\ldots,w_n)=(w_1-w_n,\ldots, w_{n-1}-w_n)$$
vanishes on $\im\Delta_n$, and  hence descends to an isomorphism $ \overline f\colon \coker \Delta_n\to W^{n-1}$ satisfying $\overline f\circ \pi = f$, where $\pi\maps{ W^n}{ \coker \Delta_n}$ is the canonical map.  Applying this isomorphism, we obtain 
\begin{align*}
&\widetilde V(\calA_1,\ldots, \calA_{n-1},\calC_1)\\
&\qquad = c V\big(\iota_1 (A_1)[n-1],\ldots,\iota_{n-1}(A_{n-1})[n-1],\Delta_{n-1}(-C_1)[n-1]\big),
\end{align*}
 where here and in the following $c$ is a positive constant that may differ from line to line, but never depends on the  convex bodies $A_1,\ldots,  A_{n-1}$. Similarly, the map $ \varphi_{1}\maps{ W^{n-1}}{ W^{n-2}\times \RR}$ given by 
 $$ \varphi_1(w_1,\ldots,w_{n-1}) = \big(w_1-w_{n-1}, \ldots, w_{n-2}-w_{n-1}, -\xi_1(w_{n-1})\big).$$
satisfies $\ker\varphi_1=\Delta_{n-1}(\ker\xi_1)$ and induces thus an isomorphism
$$\overline{\varphi_1}\maps{W^{n-1}/\Delta_{n-1}(\ker\xi_1)}{W^{n-2}\times \RR}$$ with $\overline{\varphi_1}\circ \pi_1 = \varphi_1$, where $\pi_1$ is the canonical projection.  Using Lemma~\ref{lem:volumeProjection}, we conclude that
\begin{align*}
&\widetilde V(\calA_1,\ldots, \calA_{n-1},\calC_1)\\
&\qquad= c V\big(\iota_1(A_1)[n-1],\ldots, \iota_{n-1}(A_{n-2})[n-1], (\Delta_{n-2}\times L_1)(-A_{n-1})[n-1]\big).
\end{align*} 
This finishes the proof in the case $k=1$.

Second, assume $k>1$ and that the claim holds for $k-1$. Then 
\begin{align*}
&\widetilde V(\calA_1,\ldots, \calA_{n-k}, \calC_1,\ldots, \calC_k)\\
&\qquad=  c  V\left(\iota_1(A_1)[n-1],\ldots, \iota_{n-k}(A_{n-k})[n-1], (\Delta_{n-k}\times L_{k-1})(-C_{k})[n-1]\right).
\end{align*}
We define $\varphi_k\colon W^{n-k}\times \RR^{k-1}\to W^{n-k-1}\times \RR^{k}$
by 
$$ \varphi_k(w_1,\ldots, w_{n-k}, x)= \big(w_1-w_{n-k}, \ldots , w_{n-k-1}-w_{n-k}, -L_{k-1}(w_{n-k})+x, -\xi_k(w_{n-k})\big).$$
Since $\ker\varphi_k=\Delta_{n-k}\times L_{k-1}(\ker\xi_k)$, we obtain an  isomorphism
$$\overline{\varphi_k}\maps{ \left(W^{n-k}\times \RR^{k-1}\right)\Big/\big(\Delta_{n-k}\times L_{k-1}(\ker\xi_k)\big)}{W^{n-k-1}\times \RR^{k}}$$
with $\overline{\varphi_k}\circ \pi_k = \varphi_k$, $\pi_k$ being the canonical projection, and conclude that 
\begin{align*}& \widetilde V(\calA_1,\ldots, \calA_{n-k}, \calC_1,\ldots, \calC_k)\\
 &\quad =  c V\left(\iota_{1}(A_1)[n-1],\ldots, \iota_{n-k-1}(A_{n-k-1})[n-1], (\Delta_{n-k-1}\times L_k)(-A_{n-k})[n-1]\right),
\end{align*}
as desired.
\end{proof}

\begin{theorem}
\label{thm:AFfunctionals}  Let $W$ be an $n$-dimensional real vector space, $n\geq 3$,  with a fixed positive Lebesgue measure. For linear functionals $\xi_1,\ldots, \xi_{n-2}\in W^*$, consider the graphing map $\b \xi\maps{ W}{ W\times \RR^{n-2}}$,
$$\b\xi(w)=\big(w,\xi_1(w),\ldots,\xi_{n-2}(w)\big).$$
Then for all convex bodies $A_1,\ldots, A_{n-1}, B_1,\ldots, B_{n-1}\subset W$ with $ B_1,\ldots, B_{n-1}$ centrally symmetric  we have
\begin{align*} 
&V \big(A_1,\ldots, A_{n-1},  \b\xi (B_1), \ldots, \b\xi(B_{n-1})\big)^2  \\
&\quad \geq  V \big(A_1,\ldots, A_{n-1},  \b\xi (A_1), \ldots, \b\xi(A_{n-1})\big) V\big(B_1,\ldots, B_{n-1},  \b\xi (B_1), \ldots, \b\xi(B_{n-1})\big),
 \end{align*}
where the mixed volume is taken in $W\times\RR^{n-2}$.
\end{theorem}

\begin{proof}  Observe that by continuity $\xi_1,\ldots, \xi_{n-2}\in W^*$ may be assumed to be linearly independent.
Take any $A,B,C_1,\dots,C_{n-2}\in\calK(W)$ with $C_i\subset \ker \xi_i$ centrally symmetric and $\dim C_i=n-1$ for $i=1,\ldots, n-2$,  and write $\calA=(A[n-1])$, $\calB=(B[n-1])$, and $\calC_i=(C_i[n-1])$, $i=1,\dots, n-2$. By Lemma \ref{pro:MVfunctionals} (used for $k=n-2$) we have, for some $c>0$,
\begin{align*}
\widetilde V(\calA,\calB,\calC_1,\dots,\calC_{n-2})=  cV\big(A[n-1],\b\xi(-B)[n-1]\big).
\end{align*}
Polarizing both sides and employing the assumption of central symmetry, we may use   inequality \eqref{eq:daf2} which at once yields the result.
\end{proof}

\bibliographystyle{abbrv}
\bibliography{ref_papers,ref_books}

\end{document}